\documentclass[graybox]{svmult}

\usepackage{mathptmx}       % selects Times Roman as basic font; note that this font does
\usepackage{helvet}         % selects Helvetica as sans-serif font
\usepackage{courier}        % selects Courier as typewriter font
\usepackage{graphicx}              % standard LaTeX graphics tool when including figure files

\usepackage{siunitx}
\usepackage{amsbsy}
\usepackage{amsmath,bm}
\usepackage{amssymb}
\usepackage{amsfonts}
\usepackage{multirow}
\usepackage[bottom]{footmisc}% places footnotes at page bottom
\usepackage{cite}
\usepackage{url}

\spnewtheorem{assumption}{Assumption}{\bfseries}{\itshape}

\begin{document}

\title*{Global convergence of iterative solvers for problems of nonlinear magnetostatics}
% Use \titlerunning{Short Title} for an abbreviated version of
% your contribution title if the original one is too long
\author{Herbert Egger% \inst{1}\inst{2}
\and
Felix Engertsberger% \inst{2}
\and
Bogdan Radu% \inst{1}
}

% Use \authorrunning{Short Title} for an abbreviated version of
% your contribution title if the original one is too long
\institute{Herbert Egger, Bogdan Radu \at Johann Radon Institute for Computational and Applied Mathematics, Linz, Austria,\\
\email{herbert.egger@ricam.oeaw.ac.at,bogdan.radu@ricam.oeaw.ac.at}
  \and Felix Engertsberger \at Institute for Numerical Mathematics, Johannes Kepler University, Linz, Austria,\\
  \email{felix.engertsberger@jku.at}}

%
% Use the package "url.sty" to avoid
% problems with special characters
% used in your e-mail or web address
%
\maketitle

\abstract{We consider the convergence of iterative solvers for problems of nonlinear magnetostatics. Using the equivalence to an underlying minimization problem, we can establish global linear convergence of a large class of methods, including the damped Newton-method, fixed-point iteration, and the Ka\v{c}anov  iteration, which can all be  interpreted as generalized gradient descent methods. Armijo backtracking is considered for an adaptive choice of the stepsize. The general assumptions required for our analysis cover inhomogeneous, nonlinear, and anisotropic materials, as well as permanent magnets. The main results are proven on the continuous level, but they  carry over almost verbatim to various approximation schemes, including finite elements and isogeometric analysis,  leading to bounds on the iteration numbers, which are independent of the particular discretization. The theoretical results are illustrated by numerical tests for a typical benchmark problem.}

\abstract*{We consider the convergence of iterative solvers for problems of nonlinear magnetostatics. Using the equivalence to an underlying minimization problem, we prove global linear convergence of a large class of methods, including the damped Newton-method, fixed-point iteration, and the Ka\v{c}anov  iteration, which can all be  interpreted as generalized gradient descent methods. Armijo backtracking is considered for an adaptive choice of the stepsize. The general assumptions required for our analysis cover inhomogeneous, nonlinear, and anisotropic materials, as well as permanent magnets. The main results are proven on the continuous level, but they  carry over almost verbatim to various approximation schemes, including finite elements and isogeometric analysis,  leading to bounds on the iteration numbers, which are independent of the particular discretization. The theoretical results are illustrated by numerical tests for a typical benchmark problem.}

\section{Introduction}
% % motivation
Systems of nonlinear magnetostatics and -quasistatics arise in the modelling of various high-power low-frequency applications, e.g., electric machines or power transformers. The correct representation of inhomogeneous and nonlinear material behavior is a key ingredient for the accuracy of the mathematical descriptions. For the systematic optimization and control of such devices, the resulting nonlinear partial differential equations have to be solved multiple times which requires accurate, robust and reliable solvers. 
%
%Scope
We are interested in the efficient numerical solution of such systems of nonlinear magnetostatics. As a model problem, we consider~\cite{engertsberger:Meunier2008}
\begin{alignat}{4}
\operatorname{curl} \mathbf{h} &= \mathbf{j} \quad &\text{in } \Omega, \qquad && \mathbf{h} &=  \partial_\mathbf{b} w(\mathbf{b}), \label{engertsberger:eq:1}\\
\operatorname{div} \mathbf{b} &= 0 \quad & \text{in } \Omega, \qquad && \mathbf{b} \cdot \mathbf{n} &= 0 \quad \text{on } \partial\Omega. \label{engertsberger:eq:2}
\end{alignat}
Here $\mathbf{b}$ and $\mathbf{h}$ are the magnetic flux density and field intensity, and $\mathbf{j}=\operatorname{curl} \mathbf{h_s}$ is the prescribed current density, with $\mathbf{h_s}$ a representing source field. 
%
% material law
The material law (\ref{engertsberger:eq:1}b) relates, locally at every point $x$ in the domain, the flux $\mathbf{b}$ and field $\mathbf{h}$ via the derivative of the magnetic energy density $w(\mathbf{b})$. This relation allows to model  inhomogeneous, nonlinear, and anisotropic anhysteretic material response~\cite{engertsberger:Silvester1991}. 
%For optimization or control of such devices, the underlying mathematical models have to be solved multiple times which requires accurate, robust and reliable solvers.
%

\bigskip
\noindent 
\textbf{Numerical solution.}
On a simply connected domain $\Omega$,  one can enforce (\ref{engertsberger:eq:2}a)--(\ref{engertsberger:eq:2}b) by writing $\mathbf{b} = \operatorname{curl} \mathbf{a}$ via a magnetic vector potential $\mathbf{a}$ with homogeneous boundary conditions; see e.g.~\cite{engertsberger:Meunier2008}. 
The weak form of the remaining equations~(\ref{engertsberger:eq:1}a)--(\ref{engertsberger:eq:1}b) reads
\begin{align} \label{engertsberger:eq:var}
\langle  \partial_\mathbf{b} w(\operatorname{curl} \mathbf{a}), \operatorname{curl} \mathbf{v} \rangle_\Omega &= \langle \mathbf{h_s}, \operatorname{curl} \mathbf{v} \rangle_\Omega \qquad \forall \mathbf{v} \in V_0.
\end{align}
Here $V_0$ is an appropriate subspace of $H(\operatorname{curl};\Omega)$ incorporating the required gauging and boundary conditions and $\langle \cdot,\cdot\rangle_\Omega$ is the $L^2$-scalar product.
For the solution of the nonlinear variational problem \eqref{engertsberger:eq:min}, we consider iterative methods of the form 
\begin{align} \label{engertsberger:eq:iter}
\mathbf{a}^{n+1} = \mathbf{a}^n + \tau^n \mathbf{\delta a}^n, \qquad n \ge 0,
\end{align}
with appropriate stepsize $\tau^n > 0$ chosen below and search direction $\mathbf{\delta a}^n$ defined as the solution of a related linear variational problem
\begin{align}
    \langle \nu^n \operatorname{curl} \mathbf{\delta a}^n  , \operatorname{curl} \mathbf{v} \rangle_\Omega = \langle   \mathbf{h_s} -  \partial_\mathbf{b} w(\operatorname{curl} \mathbf{a}^n), \operatorname{curl} \mathbf{v} \rangle_\Omega \qquad \forall\,\mathbf{v} \in V_0. 
    \label{engertsberger:eq:update}
\end{align}
Particular choices for the generalized reluctivity tensor $\nu^n$ and stepsize $\tau^n=\tau>0$ 
lead to some well-known schemes discussed in the literature, e.g., 
\begin{itemize}
\item $\nu^n = \bar \nu > 0$ yields the fixed-point iteration~\cite{engertsberger:Dlala2008}; 
\item $\nu^n$ = $ \partial^2_{\mathbf{bb}} w(\operatorname{curl} \mathbf{a}^n)$ leads to the (damped) Newton method~\cite{engertsberger:Fujiwara2005,engertsberger:Borghi2004}; 
\item $\nu^n = \nu_{ch}(|\operatorname{curl} \mathbf{a}^n|)$ results in the Ka\v{c}anov iteration~\cite{engertsberger:Lomonova2019,engertsberger:Han1997}.
\end{itemize}
Note that the latter choice implicitly assumes that the material is isotropic, in which case (\ref{engertsberger:eq:1}b) can be expressed as $\mathbf{h} = \nu_{ch}(|\mathbf{b}|) \mathbf{b}$ with chord reluctivity $\nu_{ch}$.

% scope
\bigskip
\noindent 
\textbf{Main contributions.}
In this paper, we study the convergence of iterative methods of the form~\eqref{engertsberger:eq:iter}--\eqref{engertsberger:eq:update}.  
A main ingredient for our analysis is to view \eqref{engertsberger:eq:var} as the optimality system for an underlying convex minimization problem
\begin{align}
    \min_{\mathbf{a} \in V_0} \int_{\Omega}  w(\operatorname{curl} \mathbf{a}) - \mathbf{h_s} \cdot \operatorname{curl} \mathbf{a} \, dx.
\label{engertsberger:eq:min}
\end{align}
As a consequence, the iteration \eqref{engertsberger:eq:iter}--\eqref{engertsberger:eq:update} can be interpreted as a generalized gradient descent algorithm. Convergence in finite-dimensions can thus be deduced from well-known results of optimization theory; see e.g., \cite[Theorem 3.2]{engertsberger:Nocedal2006}. 
By a careful modification of the arguments, we here obtain a corresponding convergence result in infinite dimensions. A related analysis was recently carried out in~\cite{engertsberger:Heid2023} in a more general context and for a specific stepsize choice. 
We here prove global convergence of the iteration \eqref{engertsberger:eq:iter}--\eqref{engertsberger:eq:update} under general assumptions on the generalized reluctivities $\nu^n$ and the stepsize $\tau^n$ which, in particular, cover all methods discussed above. 
As one appropriate admissible stepsize rule, we consider Armijo backtracking  \cite{engertsberger:Nocedal2006} \begin{align} \label{engertsberger:eq:armijo}
\tau^n &= \max\big\{\tau=\rho^k: k \ge 0 \quad \text{such that} \\ 
&\qquad \qquad  \Phi(\mathbf{a}^n+\tau \mathbf{\delta a}^n) \le \Phi(\mathbf{a}^n)  + \sigma\,\tau\, \langle \partial_\mathbf{b} \phi(\operatorname{curl} \mathbf{a}^n), \operatorname{curl} \mathbf{\delta a}^n\rangle_\Omega\big\} \notag
\end{align}
with parameters $0<\rho<1$ and $0 < \sigma < 1/2$.  Here and below, we use
\begin{align}
\Phi(\mathbf{a})=\int_\Omega \phi(\operatorname{curl} \mathbf{a}) \, dx, \qquad 
\phi(\operatorname{curl} \mathbf{a})=w(\operatorname{curl} \mathbf{a}) - \mathbf{h_s} \cdot \operatorname{curl} \mathbf{a},
\end{align} 
to abbreviate the total magnetic energy and energy density.
Notice that the Armijo rule \eqref{engertsberger:eq:armijo} eventually chooses $\tau^n=1$ in which case \eqref{engertsberger:eq:iter}--\eqref{engertsberger:eq:update} with $\nu^n = \partial_\mathbf{b} w(\operatorname{curl} \mathbf{a}^n)$ leads to the full Newton step. Hence fast convergence can be expected.

\bigskip 
\noindent
\textbf{Generalizations.}
Our analysis is carried out on the infinite-dimensional level. The results, however, carry over almost verbatim to appropriate Galerkin approximations, like FEM or IgA, and yield convergence factors for the nonlinear iterations that are independent of the discretization parameters, e.g., the polynomial degree or the meshsize.  
Similar results can also be obtained for other formulations of nonlinear magnetostatics, e.g. the reduced scalar potential formulation; see \cite{engertsberger:Engertsberger2023} for details.

\bigskip 
\noindent 
\textbf{Outline.}
In Section \ref{engertsberger:section:vector:potential} we introduce our notation and main assumptions and then formally state the model problem under investigation and comment on its well-posedness. In Section~\ref{engertsberger:section:global:convergence}, we provide the details about the iterative methods under consideration, and we then state and prove our main convergence result.  
Computational results are presented in Section~\ref{engertsberger:section:numerical:illustration} for illustration of our theoretical findings. 
%%%%%%%%%%%%%%%%%%%%%%%%%%%%%%%%%%%%%%
\section{Preliminaries}
\label{engertsberger:section:vector:potential}
Let $\Omega \subset \mathbb{R}^d$, $d=2,3$ be a bounded and Lipschitz domain. We denote by $L^2(\Omega)$, $H^1(\Omega)$, and $H(\operatorname{curl};\Omega) \subset L^2(\Omega)^d$ the usual function spaces arising in computational electromagnetics~\cite{engertsberger:Monk}. 
We further write $H_0^1(\Omega)$ and $H_0(\operatorname{curl};\Omega)$ for the subspaces of functions with corresponding homogeneous boundary conditions. 
In three space dimensions, we split $H_0(\operatorname{curl}) = \nabla H_0^1 \oplus V_0$ with a closed subspace $V_0 \subset H_0(\operatorname{curl};\Omega)$ being independent of gradient fields. One particular choice is the orthogonal complement $V_0 = \{\mathbf{v} \in H_0(\operatorname{curl};\Omega) : \langle \mathbf{v}, \nabla \psi\rangle_\Omega = 0 \; \forall \psi \in H_0^1(\Omega)\}$, but also other choices leading to a stable splitting are admissible; see e.g.~\cite{engertsberger:Monk,engertsberger:Meunier2008}. 
In two dimensions, we simply set $V_0:=H_0(\operatorname{curl};\Omega) \simeq H_0^1(\Omega)$, assuming that $\Omega$ is simply connected. 
We can now introduce the main assumptions for our analysis.
\begin{assumption} \label{engertsberger:ass:1}
$\Omega \subset \mathbb{R}^d$, $d=2,3$ is a bounded Lipschitz domain and simply connected. 
The energy density $ w : \Omega \times \mathbb{R}^d \to \mathbb{R}$ is piecewise smooth in the first variable and satisfies, for every $x \in \Omega$, the following conditions:
\begin{itemize}
\item $ w(x,\cdot) \in C^2(\mathbb{R}^d)$,
\qquad  $| \partial_b w(x,0)| \le C$, 
\smallskip
\item 
$\gamma\,|\xi|^2 \leq \langle  \partial^2_{\mathbf{bb}} w(x,\eta) \, \xi, \xi \rangle \leq L \,|\xi|^2 
\quad \forall \xi,\eta \in \mathbb{R}^d$.
\end{itemize}
The current density satisfies $\mathbf{j}=\operatorname{curl} \mathbf{h_s}$ for some $\mathbf{h_s} \in H(\operatorname{curl};\Omega)$.
The space $V_0$ is chosen as indicated above, such that $\|\mathbf{v}\|_{L^2(\Omega)} \le C_\Omega \|\operatorname{curl} \mathbf{v}\|_{L^2(\Omega)}$ for all $\mathbf{v} \in V_0$.
\end{assumption}
For ease of notation, we will usually drop the dependence on the spatial coordinate~$x$ and simply write $w(\mathbf{b})$ instead of $w(x,\mathbf{b})$ in the following, as we did in the introduction.
As a first preliminary result, let us briefly comment on the equivalence of the problems \eqref{engertsberger:eq:var} and \eqref{engertsberger:eq:min}, as well as the well-posedness of both formulations.
\begin{lemma} \label{engertsberger:lemma:wellposedness}
Let Assumption~\ref{engertsberger:ass:1} be valid. Then the nonlinear variational problem \eqref{engertsberger:eq:var}  has a unique solution which is also the unique minimizer of \eqref{engertsberger:eq:min}.
\end{lemma}
\begin{proof}
The conditions of Assumption~\ref{engertsberger:ass:1} guarantee the strong coercivity and quadratic boundedness of the energy density and further imply the bounds
\begin{align}
    \gamma\,|\mathbf{u}-\mathbf{z}|^2 \le \langle  \partial_\mathbf{b} w(x,\mathbf{u}) -  \partial_\mathbf{b} w(x,\mathbf{z}), \mathbf{u}-\mathbf{z} \rangle \le L\,|\mathbf{u}-\mathbf{z}|^2. 
    \label{engertsberger:eq:dbw}
\end{align}
Then \eqref{engertsberger:eq:iter}--\eqref{engertsberger:eq:update} with $\nu = 1$ and $\tau^n = \tau>0$ can be interpreted as Zarantonello's fixed-point iteration~\cite{engertsberger:Zeidler19902B}, and contraction can be established for all $0 < \tau < \tau^*$ with $\tau^*$ depending only on $\gamma$, $L$, and $C_\Omega$; see \cite{engertsberger:Heise1994} for a related analysis. This already ensures existence of a unique solution $\mathbf{a} \in V_0$ to \eqref{engertsberger:eq:var}. 
It is further not difficult to see that this system amounts to the first order optimality conditions for the convex minimization problem \eqref{engertsberger:eq:min}, which yields the second assertion of the lemma. \qed
\end{proof}

\section{Main results}
\label{engertsberger:section:global:convergence}

The fixed-point iteration used to establish the existence of a unique solution 
%in Lemma~\ref{engertsberger:lemma:wellposedness} 
can be viewed as a gradient descent method applied to the minimization problem \eqref{engertsberger:eq:min}. The estimates 
\eqref{engertsberger:eq:dbw} further imply that the functional $\Phi(\mathbf{a})$ over which is minimized is strongly convex and coercive. We will now use these observations to prove convergence of the iteration \eqref{engertsberger:eq:iter}--\eqref{engertsberger:eq:update} for appropriate stepsizes $\tau^n$ and more general choices of $\nu^n$ satisfying the following conditions.
\begin{assumption}
\label{engertsberger:ass:2}
   The generalized reluctivity tensors $\nu^n : \Omega \times \mathbb{R}^d \rightarrow \mathbb{R}^{d\times d}$ are symmetric and uniformly elliptic and bounded, i.e., one has
    \begin{align}
    \nu^n(x,\eta) = \nu^n(x,\eta)^T 
    \quad \text{and} \quad 
    \alpha \,|\xi|^2 \leq \langle \nu^n(x,\eta) \, \xi , \xi \rangle \leq \beta \,|\xi|^2 \label{engertsberger:eq:nu}
    \end{align}
    for all $x \in \Omega$, and $\xi,\eta \,\in \mathbb{R}^d$ with some positive constants $\alpha,\beta>0$. 
    % Moreover, $0<\tau_* \le \tau^n \le \tau^* < \frac{2\alpha}{L}$.
\end{assumption}
We can now state and prove the main theorem of our manuscript.
\begin{theorem}
\label{engertsberger:theorem:global:convergence}
  Let Assumptions \ref{engertsberger:ass:1} and \ref{engertsberger:ass:2} be valid. 
  Then for any choice $\mathbf{a}^0 \in V_0$, the iteration \eqref{engertsberger:eq:iter}--\eqref{engertsberger:eq:armijo} 
  is well-defined and the iterates $\mathbf{a}^n$ converge to the unique solution $\mathbf{a} \in V_0$ of \eqref{engertsberger:eq:min} respectively \eqref{engertsberger:eq:var}. 
  Moreover, the convergence is \emph{$r$-linear}, i.e., 
    \begin{align}
        \|\operatorname{curl} (\mathbf{a}^n - \mathbf{a})\|_{L^2(\Omega)}^2 \leq C\,q^n \|\operatorname{curl} (\mathbf{a}^0 - \mathbf{a})\|_{L^2(\Omega)}^2
    \end{align}
    with constant $C = \frac{L}{\gamma}$ and contraction factor $q = 1 - \rho \min\{2(1-\sigma)\frac{\alpha}{L},1\}\,\sigma\,\frac{2\,\gamma^2}{L\,\beta}$.
    % $q= 1 - \rho \min\{2(1-\sigma)\frac{\alpha}{L},1\}\,\sigma\,\frac{2\,\gamma^2}{L\,\beta}$.
\end{theorem}
By Assumption~\ref{engertsberger:ass:1}, the decay estimate also holds for the full $H(\operatorname{curl})$-norm of the error with the same contraction factor $q$ and with constant $C=\frac{L}{2\,\gamma} (1+C_\Omega)$.

\bigskip 
The remainder of this section is devoted to the proof of this result. We assume that $\mathbf{a}^n \in V_0$ for some $n \ge 0$ is already available and consider the next iterate.

\smallskip 
\noindent 
\textbf{Step 1.}
The conditions \eqref{engertsberger:eq:nu} and Assumption~\ref{engertsberger:ass:1} imply that 
\begin{align} \label{engertsberger:eq:scalar}
\langle \operatorname{curl} \mathbf{v}, \operatorname{curl} \mathbf{w} \rangle_{\nu^n} :=  \langle \nu^n \operatorname{curl} \mathbf{v}, \operatorname{curl} \mathbf{w} \rangle_\Omega
\end{align} 
defines a scalar product on $V_0$. 
Hence \eqref{engertsberger:eq:update} amounts to a linear elliptic variational problem on a Hilbert space, and existence of a unique solution $\mathbf{\delta a}^n \in V_0$ follows by the Lax-Milgram lemma \cite{engertsberger:Zeidler19902B}.  

\smallskip 
\noindent 
\textbf{Step 2.}
We next show that $\mathbf{\delta a}^n$ is a descent direction for \eqref{engertsberger:eq:min}.  For ease of notation, we abbreviate $\mathbf{b}^n = \operatorname{curl} \mathbf{a}^n$ and $\mathbf{\delta b}^n = \operatorname{curl} \mathbf{\delta a}^n$. Then by $\Phi(\mathbf{a}) = \int_\Omega \phi(\operatorname{curl} \mathbf{a}) \, dx$, $\phi(\mathbf{b}) = w(\mathbf{b}) - \mathbf{h_s} \cdot \mathbf{b}$, and the fundamental theorem of calculus, we get 
\begin{align*}
\Phi(\mathbf{a}^n &+ \tau \mathbf{\delta a}^n) - \Phi(\mathbf{a}^n) 
= \int_0^\tau \langle \partial_\mathbf{b} \phi(\mathbf{b}^n + t \mathbf{\delta b}^n), \mathbf{\delta b}^n\rangle_\Omega \, dt \\
&=  \int_0^\tau \langle \partial_\mathbf{b} \phi(\mathbf{b}^n + t \mathbf{\delta b}^n) - \partial_\mathbf{b} \phi(\mathbf{b}^n), \mathbf{\delta b}^n\rangle_\Omega + \langle \partial_\mathbf{b} \phi(\mathbf{b}^n), \mathbf{\delta b}^n\rangle_\Omega\, dt \\
&\le \frac{L \tau^2}{2} \|\mathbf{\delta b}^n\|^2_{L^2(\Omega)} - \tau \|\mathbf{\delta b}^n\|_{\nu^n}^2 
\le -\tau (1 - \frac{L \tau}{2\alpha}) \|\mathbf{\delta b}^n\|^2_{\nu^n}.
\end{align*}
Here we used \eqref{engertsberger:eq:dbw} in the first, and \eqref{engertsberger:eq:nu} for the second inequality. 
If $\mathbf{\delta b}^n = 0$, we already arrived at the solution. Otherwise, we certainly obtain some decay in the energy for any stepsize $0 < \tau < \frac{2\alpha}{L}$ sufficiently small. 

\smallskip 
\noindent 
\textbf{Step 3.}
As a consequence and noting that $\| \mathbf{\delta b}^n\|_{\nu^n} = -\langle \partial_\mathbf{b} \phi(\operatorname{curl} \mathbf{a}^n),\operatorname{curl} \mathbf{\delta a}^n\rangle_\Omega$, we see that the Armijo rule \eqref{engertsberger:eq:armijo} selects a stepsize $\tau^n$ satisfying 
\begin{align}
    0 < \tau_* \leq \tau^n \leq 1
\end{align}
with uniform lower bound given by $\tau_* = \rho\,\min\{2\,(1-\sigma)\,\alpha/L,\,1\}$.

\smallskip 
\noindent 
\textbf{Step 4.}
Similar to before, we abbreviate $\mathbf{b}^n=\operatorname{curl} \mathbf{a}^n$ and $\mathbf{b} = \operatorname{curl} \mathbf{a}$. Then using the fundamental theorem of calculus twice and $\partial_\mathbf{b} \phi(\mathbf{b}) = 0$, we see that 
\begin{align*}
\Phi(\mathbf{a}^n) - \Phi(\mathbf{a}) 
&= \int_0^1 \langle \partial_\mathbf{b} \phi(\mathbf{b} + t (\mathbf{b}^n - \mathbf{b})), \mathbf{b}^n - \mathbf{b} \rangle_\Omega \, dt \\
&= \int_0^1 \int_0^t \langle \partial_\mathbf{bb} w(\cdot)\, (\mathbf{b}^n - \mathbf{b}), \mathbf{b}^n - \mathbf{b} \rangle_\Omega \ ds\,dt.
\end{align*}
Here we used $ \Phi(\mathbf{a})=\int_\Omega \phi(\operatorname{curl}(\mathbf{a})) \, dx$ with $\phi(\mathbf{b}) = w(\mathbf{b}) - \mathbf{h_s} \cdot \mathbf{b}$. 
With the help of Assumption~\ref{engertsberger:ass:1}, we can then immediately deduce that 
\begin{align}
\frac{\gamma}{2} \|\operatorname{curl} (\mathbf{a}^n - \mathbf{a})\|_{L^2(\Omega)}^2 \le \Phi(\mathbf{a}^n) - \Phi(\mathbf{a}) \le \frac{L}{2} \|\operatorname{curl} (\mathbf{a}^n - \mathbf{a})\|_{L^2(\Omega)}^2.
\label{engertsberger:norm:energy}
\end{align}

\smallskip 
\noindent 
\textbf{Step 5.}
By employing the definition of $\mathbf{\delta a}^n$ in \eqref{engertsberger:eq:update}, adding $\partial_\mathbf{b} \phi(\operatorname{curl} \mathbf{a}) = 0$, testing in~\eqref{engertsberger:eq:update} with $\mathbf{v}=\mathbf{a} - \mathbf{a}^n$, and using the lower bounds in \eqref{engertsberger:eq:nu}, we get
\begin{align*}
    \| \operatorname{curl} \mathbf{\delta a}^n \|_{\nu^n} &= \sup_{\mathbf{v} \in V} \frac{\langle \nu^n \operatorname{curl} \mathbf{\delta a}^n, \operatorname{curl} \mathbf{v} \rangle_\Omega}{\|\operatorname{curl} \mathbf{v}\|_{\nu^n}} = \sup_{\mathbf{v} \in V} \frac{\langle \partial_\mathbf{b} \phi(\operatorname{curl} \mathbf{a}) - \partial_\mathbf{b} \phi(\operatorname{curl} \mathbf{a}^n), \operatorname{curl} \mathbf{v} \rangle_\Omega}{\|\operatorname{curl} \mathbf{v}\|_{\nu^n}}
     \\
    &\geq \frac{\gamma\,\|\operatorname{curl} \mathbf{a} - \operatorname{curl} \mathbf{a}^n\|_{L^2(\Omega)}^2}{\|\operatorname{curl} \mathbf{a} - \operatorname{curl} \mathbf{a}^n\|_{\nu^n}} \geq \frac{\gamma}{\sqrt{\beta}} \|\operatorname{curl} \mathbf{a} - \operatorname{curl} \mathbf{a}^n\|_{L^2(\Omega)},
    \end{align*}
This yields the lower bound
$\| \operatorname{curl} \mathbf{\delta a}^n \|_{\nu^n}^2 \ge \frac{2\,\gamma^2}{L\,\beta} \big(\Phi(\mathbf{a}^n) - \Phi(\mathbf{a})\big)$
for the increments.
The previous estimate, the definition of $\mathbf{\delta a}^n$ in \eqref{engertsberger:eq:update}, 
and the Armijo rule~\eqref{engertsberger:eq:armijo} lead to
\begin{align*}
\Phi(\mathbf{a}^{n+1}) - \Phi(\mathbf{a})
&=  \big(\Phi(\mathbf{a}^n) - \Phi(\mathbf{a})\big) + \big( \Phi(\mathbf{a}^{n+1}) - \Phi(\mathbf{a}^n)\big) \\
&\le \big(1 - \tau_* \sigma \frac{2 \gamma^2}{L\,\beta}\big) \, \big( \Phi(\mathbf{a}^n) - \Phi(\mathbf{a})\big).
\end{align*}
A recursive application of this inequality yields $\Phi(\mathbf{a}^n) - \Phi(\mathbf{a}) \le q^n \big( \Phi(\mathbf{a}^0) - \Phi(\mathbf{a})\big)$ with contraction factor $q$ as announced in the theorem. The final estimate then follows by simply applying the two bounds in~\eqref{engertsberger:norm:energy} another time. \qed 

\bigskip
\noindent 
\textbf{Remarks and extensions.}
Theorem~\ref{engertsberger:lemma:wellposedness} provides a global convergence result for the continuous version of the vector potential formulation of magnetostatics. 
A quick inspection of the proof reveals that the arguments carry over verbatim also to conforming Galerkin approximations of the problem, e.g., by FEM or IgA, and yield global convergence on the discrete level with the same convergence factor $q$. In particular, one obtains uniform global convergence, independent of the discretization parameters. This will also be demonstrated in our numerical tests below.

The results further remain valid for inexact Galerkin approximations resulting from numerical quadrature or domain approximations. 
For the Newton direction, one can additionally show local quadratic convergence; the convergence radius however depends on the mesh size $h$ and the polynomial degree $p$. 
A detailed analysis in these directions will be presented elsewhere. 
The above convergence analysis also applies to the  scalar potential formulation of magnetostatics; see~\cite{engertsberger:Engertsberger2023}, and may be extended to certain problems in magneto-quasistatics which may involve other nonlinearities; see~\cite{engertsberger:Dular2020} for an application involving superconductive materials.

\section{Numerical validation}
\label{engertsberger:section:numerical:illustration}

To illustrate our theoretical findings, let us now briefly report about some numerical tests. For ease of presentation, we consider a two-dimensional problem which is inspired by Benchmark~II in~\cite{engertsberger:Lomonova2019}. Details on the geometry, which is depicted in Figure~\ref{engertsberger:figure:pictures}, and on the choice of boundary conditions can be found in this reference. 

\begin{figure}[ht]
\centering
\includegraphics[trim={6.3cm 9.4cm 5.7cm 8cm},clip,width=0.34\textwidth]{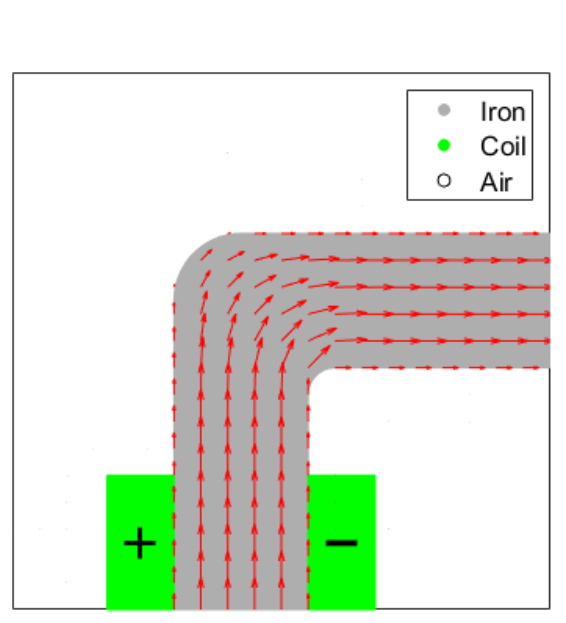}
\qquad \qquad
\includegraphics[trim={6.3cm 9.4cm 5.7cm 8cm},clip,width=0.34\textwidth]{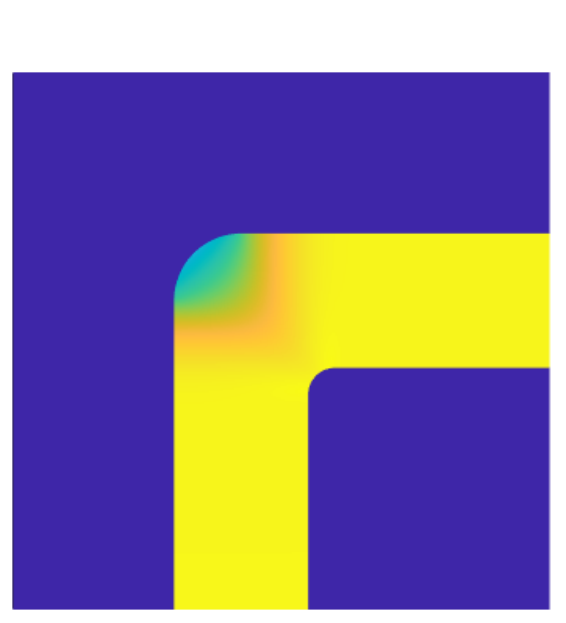}
\caption{ 
 Numerical solution of the magnetostatic problem obtained by the vector potential formulation using second-order finite elements. Left: geometric setup with iron (grey) and a coil (green) with the magnetic flux lines (red). Right: magnitude of the magnetic flux $\mathbf{b} = \operatorname{curl} \mathbf{a}$.}
\label{engertsberger:figure:pictures}
\end{figure}

\smallskip 
\noindent
\textbf{Problem data.}
For copper and air, we choose $w(\mathbf{b})=\frac{1}{2\mu_0} |\mathbf{b}|^2$ resulting in the linear relation $\mathbf{h}=\partial_\mathbf{b} w(\mathbf{b}) = \frac{1}{\mu_0} \mathbf{b}$. 
For the ferromagnetic core, the isotropic energy density $w(\mathbf{b})=\widetilde w(|\mathbf{b}|)$ is constructed as follows: First, the material data provided in the TEAM~13 problem description~\cite{engertsberger:team13} is interpolated by a strongly monotone spline function~\cite{engertsberger:Pechstein2006}, which is then integrated to obtain $\widetilde w(|\mathbf{b}|)$. 
The current density is finally chosen as $\mathbf{j}=\pm \, 6.25 \cdot 10^5 \, \text{A/m}$.
Let us note that the conditions in Assumption~\ref{engertsberger:ass:1} are thus satisfied by construction.

\smallskip 
\noindent
\textbf{Numerical results.}
For our computations, we use an isoparametric finite element method which was implemented in \textsc{Matlab}. Numerical quadrature with sufficient accuracy is used for the nonlinear terms. %
In our tests, we compare the convergence behavior of algorithm \eqref{engertsberger:eq:iter}--\eqref{engertsberger:eq:update} for different choices of the generalized reluctivities~$\nu^n$. In all cases, the Armijo rule \eqref{engertsberger:eq:armijo} is used for adaptive selection of the stepsize and as termination criteria we use that the difference in energy $\Phi(\mathbf{a}^{n+1}) - \Phi(\mathbf{a}^n) < \epsilon \| \operatorname{curl} \mathbf{\delta a}^0 \|_{\nu^0}$ is smaller than a relative factor $\epsilon = 10^{-7}$ times the norm of the first Newton step $\mathbf{\delta a}^0$.
In Table~\ref{engertsberger:tab:1}, we list details about the discretization parameters and the resulting problem dimension, and we report about the iteration numbers obtained by the iterative methods under consideration. 
\begin{table}[ht]
\centering
\setlength\tabcolsep{1.5ex}
\renewcommand{\arraystretch}{1.1}
\resizebox{0.93\linewidth}{!}{%
\begin{tabular}{c||c|c|c|c|c||c|c|c|c}
order & \multicolumn{5}{c||}{$p = 1$}  & \multicolumn{4}{c}{$p = 2$} \\
\hline \\[-0.37cm]
$h$ & $2^{-1}$ & $2^{-2}$    & $2^{-3}$  & $2^{-4}$
         & $2^{-5}$ & $2^{-1}$    & $2^{-2}$ & $2^{-3}$ & $2^{-4}$  \\
\hline
dof & $1.2k$ & $4.7k$    & $19k$  & $76k$
         & $304k$ &  $4.7k$    & $19k$  & $76k$
         & $304k$ \\
\hline \hline
%\multicolumn{10}{||c||}{iterations} \\
%\hline
fixed-point & $855$ & $859$ & $857$ & $857$
         & $-$ & $857$ & $857$& $857$& $-$  \\
\hline
Ka\v{c}anov & $45$ & $45$ & $52$ & $62$
         & $61$ & $45$ & $54$& $58$& $60$  \\
\hline
Newton & $10$ & $11$ & $11$ & $10$
         & $10$ & $10$ & $10$& $10$& $10$ 
\end{tabular}}

\caption{Discretization parameters and iteration numbers for the individual methods.}
\label{engertsberger:tab:1}
\end{table}

As predicted by Theorem~\ref{engertsberger:theorem:global:convergence}, all method yield global convergence with iteration numbers that are independent of the discretization parameters. 
The fixed-point iteration (with $\nu= 7.98 \cdot 10^4$, which was chosen by some tuning) requires most iterations, but the computational cost for each iteration can be reduced by re-using the factorization of the system matrix for determining the update direction. 
The Ka\v{c}anov iteration uses additional information about the material behavior leading to faster convergence, but a separate factorization is required in every step. 
As expected, the Newton method with line search has the smallest iteration numbers and  it is also the fastest concerning computation times.

\begin{acknowledgement}
The authors are grateful for financial support by the international FWF/DFG funded Collaborative Research Centre CREATOR (TRR361/SFB-F90). \\[-2em]
\end{acknowledgement}

% \bibliographystyle{spmpsci}
% \bibliography{reference.bib}

%%%%%%%%%%%%%%%%%%%%%%%%%%%%%%%%%%%%%%%%%%%%%%%%%%%%%%%%%%%%%%%%%%%%%%

%%%%%%%%%%%%%%%%%%%%%%%%%%%%%%%%%%%%%%%%%%%%%%%%%%%%%%%%%%%%%%%%%%%%%%

\end{document}